\documentclass[12pt]{opt2019} 
\usepackage{natbib}
\usepackage{hyperref}       
\usepackage{url}            
\usepackage{booktabs}       
\usepackage{amsfonts}       
\usepackage{nicefrac}       
\usepackage{microtype}      
\usepackage{amssymb}
\usepackage{amsmath}
\usepackage{indentfirst}
\usepackage{cite}
\usepackage{algorithm}
\usepackage{algorithmic}
\usepackage{graphicx}
\newtheorem{problem}{Problem}

\title{The Application of Multi-block ADMM on Isotonic Regression Problems}

\optauthor{\Name{Junxiang Wang} \Email{jwang40@gmu.edu}\\
  \Name{Liang Zhao} \Email{lzhao9gmu.edu}\\
  \addr George Mason University, Fairfax, VA 22030}

\begin{document}

\maketitle
\vspace{-1.5cm}
\begin{abstract}%
The multi-block ADMM has received much attention from optimization researchers due to its excellent scalability. In this paper, the multi-block ADMM is applied to solve two large-scale problems related to isotonic regression. Numerical experiments show that the multi-block ADMM is convergent when the chosen parameter is small enough and the multi-block ADMM scales well compared with baselines. 
\end{abstract}
\vspace{-0.4cm}
\section{Introduction}
\indent In recent years, the Alternating Direction Method of Multipliers (ADMM) has received considerable attention from the community of machine learning researchers. This is because it is a natural fit for wide large-scale data applications, including deep learning \citep{Wang:2019:AED:3292500.3330936}, phase retrieval \citep{wen2012alternating}, vaccine adverse event detection \citep{wang2018multi,wang2018semi} and compressive sensing \citep{chartrand2013nonconvex}. The direct extension of the classic ADMM is multi-block ADMM (i.e. ADMM with no less than three variables), which is written mathematically as follows:
\small
\begin{align*}
    &\min\nolimits_{x_1,\cdots,x_n} \sum\nolimits_{i=1}^n f_i(x_i), \ \ s.t. \sum\nolimits_{i=1}^n A_ix_i=0
    \end{align*}
    \normalsize
where $f_i:\mathbb{R}^{m_i}\rightarrow \mathbb{R}(i=1,\cdots,n)$ are convex functions, $x_i\in \mathbb{R}^{m_i}(i=1,\cdots,n)$ are vectors of length $m_i$. $A_i\in \mathbb{R}^{p\times m_i}(i=1,\cdots,n)$ are matrices. The augmented Lagrangian function  is formulated as $L_\rho(x_1,\cdots,x_n,y)=\sum\nolimits_{i=1}^n f_i(x_i)+(\rho/2)\Vert\sum\nolimits_{i=1}^n A_ix_i+y/\rho\Vert^2_2$ where $y\in \mathbb{R}^p$ is a dual variable and $\rho>0$ is a penalty parameter. The multi-block ADMM is solved by the following steps:
\small
\begin{align*}
    & x_i^{k+1}\leftarrow \arg\min\nolimits_{x_i} L_\rho(\cdots,x^{k+1}_{i-1},x_i,x^k_{i+1},\cdots) (i=1,\cdots,n)\\
    & y^{k+1}\leftarrow y^k+\rho\sum\nolimits_{i=1}^n A_ix^{k+1}_i
\end{align*}
\normalsize
A variety of related works on the convergence of the multi-block ADMM are detailed in Section \ref{sec: related work} in the Appendix.\\
\indent  While the multi-block ADMM can solve problems with linear equality constraints, it cannot directly be applied to the problems with multiple inequality constraints such as  the isotonic regression problem \citep{kyng2015fast}. In this paper, we propose new strategies based on the multi-block ADMM to address existing computational challenges in the  isotonic regression problems. The well-known isotonic regression aims to return a sequence of responses given a predictor and pre-defined order constraints, which has been addressed by many previous works. See \citep{stout2013isotonic,best1990active,barlow1972isotonic,kalai2009isotron,moon2010intervalrank,kyng2015fast} for more information. However, The main drawback is that their computational cost is very expensive when solving large-scale problems. \\
\indent To deal with the challenge of scalability, we leverge the advantage of parallel computing of the multi-block ADMM by integrating objective variables into several vectors. The following questions are addressed for the proposed multi-block ADMM frameworks on two isotonic regression problems: \textit{1. Does the multi-block ADMM converge? 2. Does the multi-block ADMM scale well?}
\vspace{-0.4cm}
\section{Isotonic Regression Problems}
\label{sec:applications}
\vspace{-0.4cm}
\begin{algorithm} 
\tiny
\caption{The Multi-block ADMM Algorithm to Solve Problem \ref{prob: smoothed  isotonic regression}} 
\begin{algorithmic}[1] 
\STATE Initialize $p$, $q$, $u$, $y_1$, $y_2$, $\rho>0$, $k=0$.
\REPEAT
\STATE Update $u^{k+1}$ in Equation \eqref{eq:isotonic update u}.
\STATE Update $p^{k+1}$ in Equation \eqref{eq:isotonic update p}.
\STATE Update $q^{k+1}$ in Equation \eqref{eq:isotonic update q}.
\STATE Update $r_1^{k+1}\leftarrow p^{k+1}-q^{k+1}+u^{k+1}$.
\STATE Update $r_{2,i}^{k+1}\leftarrow p^{k+1}_{i+1}-q^{k+1}_i(i=1,\cdots,n-2), r^{k+1}_2\leftarrow[r^{k+1}_{2,1};\cdots;r^{k+1}_{2,n-2}]$.
\STATE Update $s_1^{k+1}\leftarrow \rho(p^{k+1}-q^{k+1}-p^k+q^k)$.
\STATE Update $s_2^{k+1}\leftarrow \rho(q^k-q^{k+1})$.
\STATE Update $r^{k+1}\leftarrow \sqrt{\Vert r_1^{k+1}\Vert^2_2+\Vert r_2^{k+1}\Vert^2_2}$. $\#$ Calculate the primal residual.
\STATE Update $s^{k+1}\leftarrow \sqrt{\Vert s_1^{k+1}\Vert^2_2+\Vert s_2^{k+1}\Vert^2_2}$. $\#$ Calculate the dual residual.
\STATE Update $y_1^{k+1}\leftarrow y_1^k+\rho r_1^{k+1}$.
\STATE Update $y_2^{k+1}\leftarrow y_2^k+\rho r_2^{k+1}$.
\STATE $\nonumber k\leftarrow k+1$.
\UNTIL convergence.
\STATE Output  $p$, $q$ and $u$.
\end{algorithmic}
\label{algo:algorithm 1}
\end{algorithm}
\vspace{-0.4cm}
\subsection{The Multi-block ADMM for Smoothed Isotonic Regression}
\label{sec:smoothed isotonic regression}
\indent The classic isotonic regression is a problem to return a non-decreasing response given a predictor. However, the fitted response resembles a step function while a response is expected to be smooth and continuous in many applications \citep{sysoev2018smoothed}. To achieve this, Sysoev and Burdakov proposed a smoothed isotonic regression problem to eliminate sharp `jumps' of the
response function given a predictor
$x_i(i=1,\cdots,n)$ \citep{sysoev2018smoothed}:
\begin{problem}[Smoothed Isotonic Regression] \label{prob: smoothed  isotonic regression}
\small
\begin{align*}
&\min\nolimits_{\beta_1,\cdots,\beta_n} \sum\nolimits_{i=1}^n w_i(x_i-\beta_i)^2+\lambda\sum\nolimits_{i=1}^{n-1} (\beta_i-\beta_{i+1})^2\\
& s.t. \ \beta_1\leqslant \beta_2 \leqslant \cdots \leqslant \beta_n
\end{align*}
\end{problem}
\normalsize
where $w_i>0(i=1,\cdots,n)$ are assigned weights, $\beta_i(i=1,\cdots,n)$ are fitted predictors, and $\lambda\geq 0$ is a penalty parameter. The multi-block ADMM is applied to realize parallel computing: we introduce two vectors $p$ and $q$ of length $n-1$ such that $p_i=\beta_i(i=1,\cdots,n-1)$ and $q_i=\beta_{i+1}(i=1,\cdots,n-1)$, respectively. The problem can be reformulated as 
\small
\begin{align*}
\min\nolimits_{p,q,u} w_1(x_1-p_1)^2\!+&\! \sum\nolimits_{i=2}^{n-1} ((w_i/2)(x_i-p_i)^2\!+\!(w_i/2)(x_i-q_{i-1})^2)\!+\!w_n(x_n-q_{n-1})^2\!+\!\lambda \Vert u\Vert^2_2\\
&s.t. \ p-q+u=0,u\geq 0, p_{i+1}=q_i, \ i=1,2,\cdots,n-2
\end{align*}
\normalsize
where $p=[p_1,\cdots,p_{n-1}]$ and $q=[q_1,\cdots,q_{n-1}]$. The augmented Lagrangian is $L_\rho(u,p,q,y_1,y_2)=w_1(x_1-p_1)^2\!+\! \sum\nolimits_{i=2}^{n-1} ((w_i/2)(x_i-p_i)^2\!+\!(w_i/2)(x_i-q_{i-1})^2)\!+\!w_n(x_n-q_{n-1})^2\!+\!\lambda \Vert u\Vert^2_2+ (\rho/2)\Vert p-q+u+y_1/\rho\Vert^2_2+(\rho/2)\sum\nolimits_{i=1}^{n-2} (p_{i+1}-q_i+y_{2,i}/\rho)^2$ where $y_1=[y_{1,1},\cdots,y_{1,n-1}]$, $y_2=[y_{2,1},\cdots,y_{2,n-2}]$ and $\rho>0$. The multi-block ADMM to solve Problem \ref{prob: smoothed  isotonic regression} is shown in Algorithm \ref{algo:algorithm 1}. Each subproblem has a closed-form solution and can be implemented in parallel, which is shown as follows:\\
\textbf{1. Update $u$.}\\
\indent The variable $u$ is updated as follows:
\small
\begin{align}
    u^{k+1}&\leftarrow\arg\min\nolimits_{u} (\rho/2)\Vert p^k-q^k+u+y^k_1/\rho\Vert^2_2+\lambda\Vert u\Vert^2_2, \ s.t. \ u\geq 0.
    \label{eq:isotonic update u}\\\nonumber &\leftarrow\max((\rho(q^k-p^k)-y^k_1)/(\rho+2\lambda),0).
\end{align}
\normalsize
\textbf{2. Update $p$.}\\
\indent The variable $p$ is updated as follows:
\small
\begin{align}
\nonumber
    p_1^{k+1}&\leftarrow\arg\min\nolimits_{p_1} w_1(x_1-p_1)^2+(\rho/2)\Vert p_1-q^k_1+u^{k+1}_1+y^k_1/\rho\Vert^2_2\\\nonumber &\leftarrow(2w_1x_1+\rho q^k_1-\rho u^{k+1}_1-y^k_{1,1})/(2w_1+\rho)\\
    \nonumber
    p_i^{k+1}&\leftarrow\arg\min\nolimits_{p_i} (w_i/2)(x_i-p_i)^2+(\rho/2)\Vert p_i-q^k_i+u^{k+1}_i+y^k_i/\rho\Vert^2_2\\&+(\rho/2)\Vert p_i-q^k_{i-1}+y^k_{2,i-1}/\rho\Vert^2_2(i=2,\cdots,n-1).
    \label{eq:isotonic update p}\nonumber\\&\leftarrow(w_ix_i+\rho q^k_i-\rho u^{k+1}_i-y^k_{1,i}+\rho q^k_{i-1}-y^k_{2,i-1})/(w_i+2\rho)(i=2,\cdots,n-1).
\end{align}
\normalsize
\textbf{3. Update $q$.}\\
\indent The variable $q$ is updated as follows:
\small
\begin{align}
    \nonumber q^{k+1}_i&\leftarrow\arg\min\nolimits_{q_i} (w_{i+1}/2)(x_{i+1}-q_{i})^2+(\rho/2)\Vert p^{k+1}_i-q_i+u^{k+1}_i+y^{k}_{1,i}/\rho\Vert^2_2\\ &+(\rho/2)(p^{k+1}_{i+1}-q_i+y^k_{2,i}/\rho)(i=1,\cdots,n-2).\label{eq:isotonic update q}\\\nonumber&\leftarrow(w_{i+1}x_{i+1}+\rho p^{k+1}_i+\rho u^{k+1}_i+y^k_{1,i}+\rho p^{k+1}_{i+1}+y^k_{2,i})/(w_{i+1}+2\rho)(i=1,\cdots,n-2)\\\nonumber
    q^{k+1}_{n-1}&\leftarrow\arg\min\nolimits_{q_{n-1}} w_n(x_n-q_{n-1})^2+(\rho/2)\Vert p^{k+1}_{n-1}-q_{n-1}+u^{k+1}_{n-1}+y^{k}_{1,n-1}/\rho\Vert^2_2.\\\nonumber
    &\leftarrow(2w_n x_n+\rho p^{k+1}_{n-1}+\rho u^{k+1}_{n-1}+y^k_{1,n-1})/(2w_n+\rho).
\end{align}
\indent Due to space limit, the convergence of Algorithm \ref{algo:algorithm 1} is discussed in Section \ref{sec:convergence} in the Appendix.
\subsection{The Multi-block ADMM for Multi-dimensional Ordering}
\label{sec:Multi-dimensional Ordering}
\indent The previous smoothed isotonic regression only considers linear orders, while multi-dimensional orders are  more general in isotonic regression applications \citep{stout2015isotonic}.  A multi-dimensional order is defined in a $m-$dimensional space $Z_i=(z_{i,1},\cdots,z_{i,m})$. $Z_i\leq Z_j$ if and only if $z_{i,1}\leq z_{j,1},\cdots,z_{i,m}\leq z_{j,m}$.  It can be represented equivalently as Directed Acyclic Graph (DAG) $G=(V,E)$  where $(Z_i,Z_j)\in E$ if $Z_i\leq Z_j$ \citep{stout2015isotonic}. Formally, the multi-dimensional ordering problem is formulated as follows:
\begin{problem}[Multi-dimensional Ordering]
\label{prob: Multi-dimensional ordering}
\small
\begin{align*}
&\min\nolimits_{\alpha_1,\cdots,\alpha_n}\sum\nolimits_{i=1}^n w_i(Y_i-\alpha_i)^2\\
&s.t. \ \alpha_i\leq \alpha_j \ \text{iff} \ Z_i\leq Z_j (1\leq i,j\leq n).
\end{align*}
\end{problem}
\normalsize
where $w_i>0(i=1,\cdots,n)$ are assigned weights, $Y_i(i=1,\cdots,n)$ are predictors, and $\alpha_i(i=1,\cdots,n)$ are fitted predictors. \\
\indent To handle large-scale multi-dimensional ordering problems, we leverage the advantage of parallel computing of the multi-block ADMM to solve it in an exact form. By introducing two vectors $g$ and $h$, this problem is equivalent of
\small
\begin{align*}
    &\min\nolimits_{g,h}  W^T((Y-g)\odot(Y-g))/2+W^T((Y-h)\odot(Y-h))/2\\
    & s.t. \ E_1 g-E_2 h+v=0, \ v\geq 0,\ g=h.
\end{align*}
\normalsize
where $W=(w_1,\cdots,w_n)$,  $Y=(Y_1,\cdots,Y_n)$ and $\odot$ is the Hadamard product. $E_1\in R^{\vert E\vert\times n}$ and $E_2\in R^{\vert E\vert\times n}$ are representations of the edge set $E$: the $k$-th edge $(i,j)\in E$ means that $E_{1,k,i}=1$ and $E_{2,k,j}=1$ while $E_{1,k,p}=0(1\leq p\leq n,p\neq i)$ and $E_{2,k,q}=0(1\leq q\leq n,q\neq j)$.
The augmented Lagrangian is $L_\rho(v,g,h,y_1,y_2)=W^T((Y-g)\odot(Y-g))/2+W^T((Y-h)\odot(Y-h))/2+(\rho/2)\Vert E_1 g-E_2 h+v+y_1/\rho\Vert^2_2+(\rho/2)\Vert g-h+y_2/\rho\Vert^2_2.$ where $\rho>0$. The multi-block ADMM to solve Problem \ref{prob: Multi-dimensional ordering} is shown in Algorithm \ref{algo:algorithm 2}. Each subproblem has a closed-form solution and can be implemented in parallel, which is shown as follows:\\
\begin{algorithm}
\tiny
\caption{The Multi-block ADMM Algorithm to Solve Problem \ref{prob: Multi-dimensional ordering}} 
\begin{algorithmic}[1] 
\STATE Initialize $g$, $h$, $v$, $y_1$, $y_2$, $\rho>0$, $k=0$.
\REPEAT
\STATE Update $v^{k+1}$ in Equation \eqref{eq:Multi-dimensional update v}.
\STATE Update $g^{k+1}$ in Equation \eqref{eq:Multi-dimensional update g}.
\STATE Update $h^{k+1}$ in Equation \eqref{eq:Multi-dimensional update h}.
\STATE Update $r_1^{k+1}\leftarrow E_1 g^{k+1}-E_2h^{k+1}+v^{k+1}$.
\STATE Update $r_2^{k+1}\leftarrow g^{k+1}-h^{k+1}$.
\STATE Update $s_1^{k+1}\leftarrow \rho(E_1g^{k+1}-E_2h^{k+1}-E_1g^k+E_2h^{k})$.
\STATE Update $s_2^{k+1}\leftarrow \rho E_2(h^k-h^{k+1})$.
\STATE Update $s^{k+1}_3\leftarrow \rho(h^k-h^{k+1})$.
\STATE Update $r^{k+1}\leftarrow \sqrt{\Vert r_1^{k+1}\Vert^2_2+\Vert r_2^{k+1}\Vert^2_2}$. $\#$ Calculate the primal residual.
\STATE Update $s^{k+1}\leftarrow \sqrt{\Vert s_1^{k+1}\Vert^2_2+\Vert s_2^{k+1}\Vert^2_2+\Vert s_3^{k+1}\Vert^2_2}$. $\#$ Calculate the dual residual.
\STATE Update $y_1^{k+1}\leftarrow y_1^k+\rho r_1^{k+1}$.
\STATE Update $y_2^{k+1}\leftarrow y_2^k+\rho r_2^{k+1}$.
\STATE $\nonumber k\leftarrow k+1$.
\UNTIL convergence.
\STATE Output  $g$, $h$ and $v$.
\end{algorithmic}
\label{algo:algorithm 2}
\end{algorithm}
\textbf{1. Update $v$.}\\
\indent The variable $v$ is updated as follows:
\small
\begin{align}
    v^{k+1}&\leftarrow\arg\min\nolimits_{v} (\rho/2)\Vert E_1 g^k-E_2 h^k+v+y^k_1/\rho\Vert^2_2, \ s.t. \ v\geq 0.
    \label{eq:Multi-dimensional update v}\\\nonumber & \leftarrow\max(-E_1g^k+E_2h^k-y_1^k/\rho,0).
\end{align}
\normalsize
\textbf{2. Update $g$.}\\
\indent The variable $g$ is updated as follows:
\small
\begin{align}
\nonumber g^{k+1}&\leftarrow\arg\min\nolimits_{g} W^T((Y-g)\odot(Y-g))/2+(\rho/2)\Vert E_1g-E_2h^k+v^{k+1}+y^k_1/\rho\Vert^2_2\\&+(\rho/2)\Vert g-h^k+y^k_2/\rho\Vert^2_2.\nonumber\\&\leftarrow (\text{diag}(W)+\rho E^T_1E_1+\rho I)^{-1}(\text{diag}(W)Y+\rho E^T_1E_2h^k-\rho E^T_1v^{k+1}-E^T_1y^k_1+\rho h^k-y^k_2)
    \label{eq:Multi-dimensional update g}
\end{align}
\normalsize
where $I$ is an identity matrix and diag($W$) is a diagonal matrix where the  main diagonal is $W$.\\
\textbf{3. Update $h$.}\\
\indent The variable $h$ is updated as follows:
\small
\begin{align}
\nonumber h^{k+1}&\leftarrow\arg\min\nolimits_{h} W^T((Y-h)\odot(Y-h))/2+(\rho/2)\Vert E_1g^{k+1}-E_2h+v^{k+1}+y^k_1/\rho\Vert^2_2\\&+(\rho/2)\Vert g^{k+1}-h+y^k_2/\rho\Vert^2_2.   \label{eq:Multi-dimensional update h}
\\\nonumber&\leftarrow(\text{diag}(W)+\rho E^T_2E_2+\rho I)^{-1}(\text{diag}(Y)+\rho E^T_2 E_1 g^{k+1}+\rho E^T_2 v^{k+1}+E^T_2 y^k_1+\rho g^{k+1}+y_2^k).
\end{align}
\normalsize
\indent Due to space limit, the convergence of Algorithm \ref{algo:algorithm 2} is discussed in Section \ref{sec:convergence} in the Appendix.
\vspace{-0.4cm}
\section{Experiment}
\label{sec:experiment}
\indent In this section, we validate the multi-block ADMM using simulated datasets and compare it with existing state-of-the-art methods.  All experiments were conducted on a 64-bit machine with Intel(R) core(TM)processor (i7-6820HQ CPU@ 2.70GHZ) and 16.0GB memory.
\vspace{-0.3cm}
\subsection{Data Generation and Parameter Settings}
Due to space limit, data and parameters are detailed in Section \ref{sec:data generation} in the Appendix.
\vspace{-0.3cm}
\subsection{Baselines}
\indent Two methods Smoothed Pool-Adjacent-Violators (SPAV) \citep{sysoev2018smoothed} and Interior Point Method (IPM) \citep{kyng2015fast} are used for comparison. The details can be found in Section \ref{sec: baselines} in the Appendix.
\subsection{Experimental Results}
\indent In this section, the experimental results on two problems are explained in detail. \\
\textbf{1. Does the multi-block ADMM converge?} Figure \ref{fig: convergence} illustrates the convergence properties of the multi-block ADMM when $n=1000$ on the smoothed isotonic regression problem and the multi-dimensional ordering problem. Two choices of $\rho=0.1$ and $\rho=10$ are shown on Figure \ref{fig: convergence}. Overall, Figure \ref{fig: convergence} (a)-(c) shows that the multi-block ADMM converges while Figure \ref{fig: convergence}(d) shows the divergence: $r$ and $s$ drop drastically at the beginning and then decrease smoothly through the end in the Figures \ref{fig: convergence} (a)-(c); however, Figure \ref{fig: convergence}(d) displays a surge of $r$. Moreover, when $\rho=0.1$, $r$ is located above $s$ while when $\rho=10$ the situation is the opposite. Obviously, the multi-block ADMM can obtain a reasonable solution within tens of iterations as long as $\rho$ is small and hence it is suitable for large-scale optimization.
\vspace{-0.4cm}
\begin{figure}[h]
  \centering
  \begin{minipage}
  {0.24\linewidth}
  \centerline{\includegraphics[width=\textwidth] {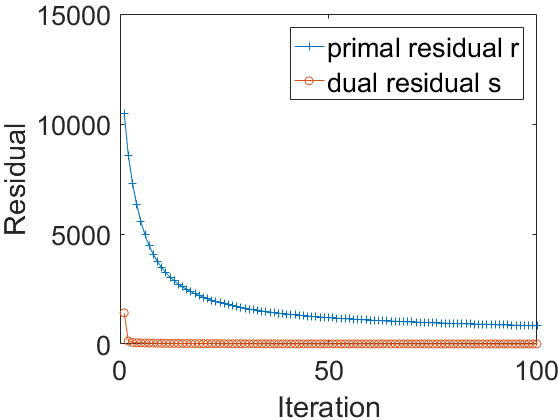}}
  \centerline{(a).$\rho=0.1$ on the}  \centerline{ smoothed isotonic} \centerline{regression problem.}
  \end{minipage}
\hfill
  \begin{minipage}
  {0.24\linewidth}
  \centerline{\includegraphics[width=\textwidth] {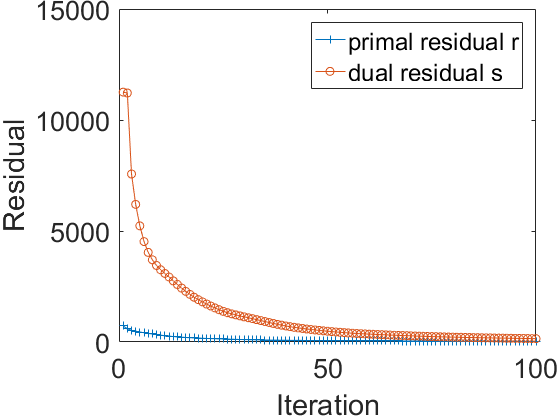}}
  \centerline{(b).$\rho=10$ on the}  \centerline{ smoothed isotonic} \centerline{regression problem.}
  \end{minipage}
  \hfill
    \begin{minipage}
  {0.24\linewidth}
  \centerline{\includegraphics[width=\textwidth] {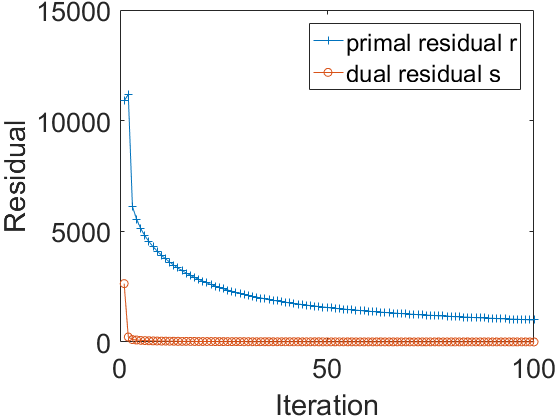}}
  \centerline{(c).$\rho=0.1$ on the} \centerline{multi-dimensional} \centerline{ ordering problem.}
  \end{minipage}
\hfill
  \begin{minipage}
  {0.24\linewidth}
  \centerline{\includegraphics[width=\textwidth] {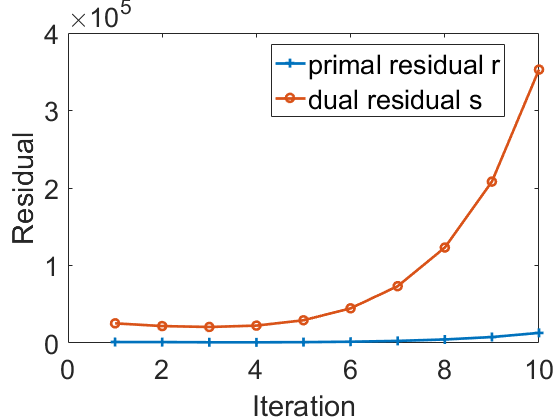}}
  \centerline{(c).$\rho=10$ on the} \centerline{multi-dimensional} \centerline{ ordering problem.}
  \end{minipage}
  \caption{Convergence of the multi-block ADMM when $n=1000$ on two problems.}
  \vspace{-0.5cm}
  \label{fig: convergence}
\end{figure}
\\
\indent\textbf{2. Does the multi-block ADMM scale well?}
As Figure \ref{fig: scalability} shown, the running time of the multi-block ADMM increases linearly with the number of observations on two problems. In Figure \ref{fig: scalability}(a), the SPAV is more efficient than the multi-block ADMM when the number of observations $n$ is less than 20,000 but needs more time since then; as for Figure \ref{fig: scalability}(b), the multi-block ADMM is more efficient than the IPM no matter how many observations there are.
\vspace{-0.4cm}
\small
\begin{figure}[h]
  \centering
  \begin{minipage}
  {0.35\linewidth}
  \centerline{\includegraphics[width=\textwidth] {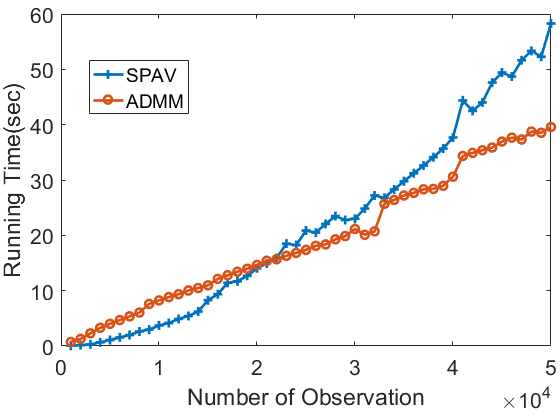}}
  \centerline{(a).Scalability on the smoothed}
  \centerline{isotonic regression problem.}
  \end{minipage}
\hfill
  \begin{minipage}
  {0.35\linewidth}
  \centerline{\includegraphics[width=\textwidth] {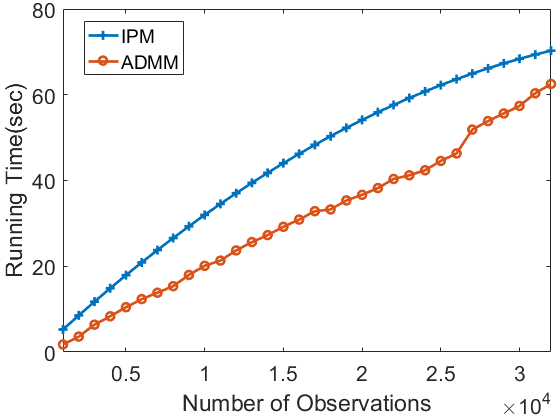}}
  \centerline{(b).Scalability on the} \centerline{multi-dimensional ordering problem.}
  \end{minipage}
 
  \caption{Scalability of the multi-block ADMM on two problems.}
    \label{fig: scalability}
  \end{figure}
  \normalsize
  \vspace{-1cm}
\section{Conclusion}
\vspace{-0.3cm}
\indent The multi-block ADMM is an interesting topic in the optimization community in recent years. In this paper, we apply the multi-block ADMM to two problems related to isotonic regression: smoothed isotonic regression and multi-dimensional ordering. Most existing methods are not efficient enough to run on large-scale datasets. However, the main advantage of the multi-block ADMM is parallel computing and hence it is scalable to large datasets. We find that the multi-block ADMM  converges when $\rho$ is small and its running time increases linearly with the scalability of observations.
\label{sec:conclusion}
\bibliographystyle{plain}
\bibliography{nips_2019}
\newpage
\textbf{\large Appendix}
\begin{appendix}
\section{The Convergence of Algorithms \ref{algo:algorithm 1} and \ref{algo:algorithm 2}}
\label{sec:convergence}
\indent In addition to the advantage of parallel computing, the convergence of the  Algorithm \ref{algo:algorithm 1} can be guaranteed by the following theorem:
\begin{theorem}
There exists $d>0$ such that if $\rho<d$, the Algorithm \ref{algo:algorithm 1} converges globally to the optimal point with sublinear convergence rate.
\end{theorem}
\begin{proof}
\indent Because the objective function is strongly convex in $p$ and $q$, all assumptions in \citep{lin2014convergence} are  satisfied and hence the Theorem 3.2 and Theorem 4.2 in \citep{lin2014convergence} ensure the global sublinear convergence if $\rho$ is smaller than a threshold $d$.
\end{proof}
\indent As for the convergence properties of Algorithm \ref{algo:algorithm 2}, the below theorem guarantees the convergence if $\rho$ is small.
\begin{theorem}
There exists $d>0$ such that if $\rho<d$, then the Algorithm \ref{algo:algorithm 2} converges globally to the optimal point with sublinear convergence rate.
\end{theorem}
\begin{proof}
\indent The proof is the same as Theorem 1.
\end{proof}
\section{Related Work on the Multi-block ADMM}
\label{sec: related work}
\indent The multi-block ADMM was firstly studied by Chen et al. \citep{chen2016direct} when they concluded that the multi-block ADMM does not necessarily converge by giving a counterexample. He and Yuan explained why the multi-block ADMM may diverge from the perspective of variational inequality framework \citep{he2014direct}. Since then, many researchers studied sufficient conditions to ensure the global convergence of  the multi-block ADMM. For example, Robinson and Tappenden and Lin et al. imposed strong convexity assumption on the objective function $f_i(x_i)(i=1,\cdots,n)$ \citep{robinson2017flexible,lin2014convergence}. Tao and Yuan did not require all objective functions $f_i(x_i)(i=1,\cdots,n)$ to be strongly convex, but imposed full-rank assumption on $A_i(i=1,\cdots,n)$ \citep{tao2016convergence}. Lin et al. weakened the strong convexity assumption on the objective function $f_i(x_i)(i=1,\cdots,n)$ by imposing the additional cocercity assumption on $f_i(x_i)(i=1,\cdots,n)$ \citep{lin2016iteration}.  Deng et al. proved that the multi-block ADMM preserved convergence if multiple variables $x_i(i=1,\cdots,n)$ were updated in the Jacobian fashion rather than Gauss-Seidel fashion \citep{deng2017parallel}.  Lin et al. proved that the multi-block ADMM converged for any $\rho>0$ in the regularized least squares decomposition problem \citep{lin2015global}. Moreover, some work extended the multi-block ADMM into the nonconvex settings. See \citep{magnusson2016convergence,wang2015global,wang2014convergence} for more information.
\section{Data Generation and Parameter Settings}
\label{sec:data generation}
The experimental data and parameters for two applications are explained in this section. Commonly, the maximal number of iteration was set to $10,000$. We set the tolerance $\varepsilon=0.01\sqrt{n}$ according to the recommendation by Boyd et al. \citep{boyd2011distributed}. For the smoothed isotonic regression problem, we generated simulated observations $x_i(i=1,\cdots,n)$ from a uniform distribution in $(0,1000)$; we set $w_i=1(i=1,\cdots,n)$ and $\lambda=1$. For the multi-dimensional ordering problem, we generated simulated observations $Y_i(i=1,\cdots,n)$ from a uniform distribution in $(0,1000)$. $w_i(i=1,\cdots,n)$ were set to 1. $E_1$ and $E_2$ were generated from a  2-dimensional random grid graph to simulate the ordering of the 2-dimensional space.
\section{Baselines}
\label{sec: baselines}
\indent In order to test the scalability of the multi-block ADMM on two applications, two baselines were utilized for comparison: \\
\textit{1. Smoothed Pool-Adjacent-Violators(SPAV) \citep{sysoev2018smoothed}.} The SPAV algorithm is a extension of the PAV algorithm designed for the smoothed isotonic regression problem. It partitions all $\beta_i(i=1,\cdots,n)$ into many blocks according to the feasibility of inequality constraints. The iteration ends when all constraints are feasible.\\
\textit{2. Interior Point Method(IPM) \citep{kyng2015fast}.} The IPM method is an efficient algorithm to solve multi-dimensional ordering  with convergence guarantee. Its time complexity is $O(\vert E\vert^{1.5}\log^2\vert V\vert\log^2(\vert V\vert/\delta))$ where $\delta$ is a tolerance.
\end{appendix}

\end{document}